\documentclass[10pt]{amsart}
\usepackage{graphicx}
\usepackage{amsthm}
\usepackage{amsmath}
\usepackage{amssymb}
\newtheorem{thm}{Theorem}[section]

\newtheorem{cor}[thm]{Corollary}
\newtheorem{lem}[thm]{Lemma}

\theoremstyle{definition}
\newtheorem{defn}[thm]{Definition}
\theoremstyle{remark}
\newtheorem{rem}[thm]{Remark}
\theoremstyle{example}
\newtheorem{ex}[thm]{Example}

\begin{document}

\title[Degree Sequence Index Strategy]{Degree Sequence Index Strategy}%

\author[Caro and Pepper]{\small Yair Caro$^*$ and Ryan Pepper$^{**}$\\
\tiny{$^*$Dept. of Mathematics and Physics\\
University of Haifa-Oranium\\
Tivon 36006, Israel\\
yacaro@kvgeva.org.il\\$^{**}$Dept. of Computer and Mathematical Sciences\\University of Houston -- Downtown\\Houston, Texas 77002, U.S.A.\\pepperr@uhd.edu}}

\begin{abstract}
We introduce a procedure, called the Degree Sequence Index Strategy (DSI), by which to bound graph invariants by certain indices in the ordered degree sequence. As an illustration of the DSI strategy, we show how it can be used to give new upper and lower bounds on the $k$-independence and the $k$-domination numbers. These include, among other things, a double generalization of the annihilation number, a recently introduced upper bound on the independence number. Next, we use the DSI strategy in conjunction with planarity, to generalize some results of Caro and Roddity about independence number in planar graphs. Lastly, for claw-free and $K_{1,r}$-free graphs, we use DSI to generalize some results of Faudree, Gould, Jacobson, Lesniak and Lindquester.  
\end{abstract}
\maketitle
% ----------------------------------------------------------------
\section{Introduction}

All graphs considered are simple and finite.  For a graph $G=(V,E)$, we will use $n=n(G)$ to denote the order, or $|V|$, and $m=m(G)$ to denote the size, or $|E|$. Moreover, we will use the notation $\Delta(G)$ and $\delta(G)$ to denote, respectively, the maximum and minimum degrees of a graph $G$. A complete graph with $r$ vertices is denoted $K_r$ and an empty graph with $r$ vertices is denoted $E_r$. If $S$ is a subset of $V$, then we use the notation $[S]$ to denote the subgraph induced by $S$. For two graphs $G$ and $H$, we use the notation $G \cup H$ to denote their union and the notation $G+H$ to denote their join (the graph obtained by joining all possible edges between $G$ and $H$). A $j$-independent set is a set $I \subseteq V$ such that $\Delta([I]) < j$. The $j$-independence number, denoted $\alpha_j(G)$, is the cardinality of a largest $j$-independent set. This generalizes the traditional independence number since $\alpha_1(G)=\alpha(G)$. A $j$-dominating set is a set $D \subseteq V$ such that each vertex in $V-D$ has at least $j$ neighbors in $D$.  The $j$-domination number, denoted $\gamma_j(G)$, is the cardinality of a smallest $j$-dominating set.  This generalizes the traditional domination number since $\gamma_1(G)=\gamma(G)$. These concepts were introduced in \cite{25,26} and both invariants have become popular research topics.  For example, $j$-independence number is studied in \cite{28,23,21,22,19}, $j$-domination is studied in \cite{29,27,20,30}, while relationships between these invariants is studied in \cite{16,24,18,17}.  In fact, the literature is so extensive that in order to see the many more articles on these topics, it would be better to consult the textbook \cite{30} and the survey article \cite{16} which collectively capture much of what is known. The degree sequence $D$ of a graph $G$, unless stated otherwise, is assumed to be in non-decreasing order and denoted; $D=D(G)=\{\delta=d_1 \leq d_2 \leq \ldots \leq d_n=\Delta\}$. 

The goal of this paper is to introduce a general method by which to constrain NP-hard graph invariants, such as the independence and domination numbers, by use of the degree sequence.  In particular, we will show how certain indices of the ordered degree sequence can be used as upper and lower bounds for various other graph invariants.  In some instances, these will be improvements or generalizations on known bounds, while in other instances, they will lead to new bounds entirely.

\section{The General Strategy}
Given a graph $G$ with degree sequence $D=\{d_1 \leq d_2 \leq \ldots \leq d_n\}$, our goal is to find both upper and lower bounds, connected to indices from $D$, for the size of a largest (smallest) induced subgraph having a given property $P$. Let $c(G)$ be a given graph invariant of $G$.  Now, for any subset $S \subseteq V$, let $h(\{ deg(v) | v \in S \})$ be a function of the degrees in $S$ such that for any two subsets of $V$, say $X$ and $Y$, with the same cardinality, if $\sum_{v \in X} deg(v) \geq \sum_{v \in Y} deg(v)$, then $h(\{ deg(v) | v \in X \}) \geq h(\{ deg(v) | v \in Y \})$.\\

%[THIS IS NOT REALLY WHAT I WANT BUT NEED TO MAKE SURE STEP 3 IS TRUE BELOW WHILE STILL MAKING MORE GENERAL]\\

The strategy we introduce, which we call the \textbf{Degree Sequence Index Strategy}, (DSI strategy) can now be described in the following steps.
\begin{enumerate}
 
\item Let $A(P)$ be an optimal induced subgraph of $G$ with property $P$. Identify functions $f_U(G,A(P))$ and $f_L(G,A(P))$, such that one of the following is true;
\[ h(\{ deg(v) | v \in A(P) \}) + f_U(G,A(P)) \leq c(G), \]
if an upper bound on $|A(P)|$ was intended, and
\[ h(\{ deg(v) | v \in A(P) \}) + f_L(G,A(P)) \geq c(G), \]
if a lower bound on $|A(P)|$ was intended.

\item Define the indices $DSI_U(G,h,f_U,k)$ and $DSI_L(G,h,f_L,k)$ as follows;
\[ DSI_U(G,h,f_U,c) = \max\{k \in \mathbb{Z} | h(\{ d_i | i \in \{1,2,\ldots,k\} \}) + f_U(G,A(P)) \leq c(G)\}, \]
\[ DSI_L(G,h,f_L,c) = \min\{k \in \mathbb{Z} | h(\{ d_{n-i+1} | i \in \{1,2,\ldots,k\} \}) + f_L(G,A(P)) \geq c(G)\}. \]

\item Next, make the following observations;
\[ h(\{ d_i | i \in \{1,2,\ldots,|A(P)|\} \}) + f_U(G,A(P)) \leq  h(\{ deg(v) | v \in A(P) \}) + f_U(G,A(P)) \leq c(G), \]
\[ h(\{ d_{n-i+1} | i \in \{1,2,\ldots,|A(P)|\} \}) + f_L(G,A(P)) \geq  h(\{ deg(v) | v \in A(P) \}) + f_L(G,A(P)) \geq c(G). \]
 
\item Finally, since $|A(P)|$ is an integer satisfying the definitions above, we conclude; $DSI_L(G,h,f_L,c) \leq |A(P)| \leq DSI_U(G,h,f_U,c)$, as was intended.

\item After this, the optional step would be to find more easily computable approximations to the functions $f_U$ and $f_L$ (and possibly to $c(G)$), so that, for example, the $DSI_U$ and $DSI_L$ can be found in polynomial time.
\end{enumerate}

%\begin{enumerate}
%\item Let $A(P)$ denote an optimal induced subgraph of $G$ with property $P$.  Identify functions $f_U(G,A(P))$ and $f_L(G,A(P))$, such that one of the following is true;
%\[ \sum_{v \in A(P)}deg(v) + f_U(G,A(P)) \leq m(G), \]
%if an upper bound on $|A(P)|$ was intended, and
%\[ \sum_{v \in A(P)}deg(v) + f_L(G,A(P)) \geq m(G), \]
%if a lower bound on $|A(P)|$ was intended.

%\item Define $U(G,f_U)$ and $L(G,f_L)$ as follows;
%\[ U(G,f_U) = \max\{k \in \mathbb{Z} | \sum_{i=1}^k d_i + f_U(G,A(P)) \leq m(G)\}, \]
%\[ L(G,f_L) = \min\{k \in \mathbb{Z} | \sum_{i=1}^k d_{n-i+1} + f_L(G,A(P)) \geq m(G)\}. \]

%\item Next, make the following observations;
%\[ \sum_{i=1}^{|A(P)|} d_i + f_U(G,A(P)) \leq  \sum_{v \in A(P)}deg(v) + f_U(G,A(P)) \leq m(G), \]
%\[ \sum_{i=1}^{|A(P)|} d_{n-i+1} + f_L(G,A(P)) \geq  \sum_{v \in A(P)}deg(v) + f_L(G,A(P)) \geq m(G). \]
 
%\item Finally, since $|A(P)|$ is an integer satisfying the definitions above, we conclude; $L(G,f_L) \leq |A(P)| \leq U(G,f_U)$, as was intended.

%\end{enumerate}

%Once this has been done, it may be that finding optimal subgraphs with property $P$ is an NP-hard task. In this case, one can concentrate on finding approximations to the functions $f_U(G,A(P))$ and $f_L(G,A(P))$, so that the upper and lower bounds, $U(G,f_U)$ and $L(G,f_L)$, could be computed more easily. 

As is evident, the most difficulty lies in the identification of the functions from the first step, and then in finding approximations to those functions for practicality.  In the next section, we give an example of this process, which we will elaborate on for much of the remainder of the paper. In a later section, we will give an example using a different graph property, giving some feeling for the generality of the DSI strategy.

\section{Application to Independence}

The monotonicity condition imposed on the function $h$ is suggestive, and leads to our first concrete example.  Namely, we identify $h(\{ deg(v) | v \in S \}) = \sum_{v \in S} deg(v)$.  Also, we will use the number of edges, or size, of $G$ as our graph invariant $c(G)=m(G)$.  Our property $P$ is that of being a $j$-independent set, so that $A(P)$ is a maximum $j$-independent set and we want to constrain $|A(P)|=\alpha_j(G)$.  Thus, in this section, we will apply the DSI strategy to find upper and lower bounds for the $j$-independence number. This problem is well motivated since calculating the $j$-independence number exactly is a computationally difficult problem \cite{15,14,13}. In some cases, we will show how these inequalities give improvements or generalizations on known results, or new results entirely.  Finally, we will consider the extreme cases where these newly discovered upper and lower bounds are sharp, as well as where they can be very poor approximations. 

The annihilation number of a graph was introduced by Pepper in \cite{3,4} -- where it was shown to be an upper bound on the independence number.  The characterization of equality for this upper bound was addressed in \cite{2}.  While reading the proof of this upper bound, Fajtlowicz formulated the definition presented below, recognizing that it also led to an upper bound on the independence number.  In \cite{5,4}, Pepper shows that the original definition is equivalent to the one presented below -- which for simplicity, and relevance to this paper, is the only one we give.

\begin{defn}\label{ann_def}
Let $D=\{d_1 \leq d_2 \leq \ldots \leq d_n\}$ be the degree sequence of a graph $G$.  The \textit{annihilation number} of $G$, denoted $a=a(G)$, can be defined by the equation: 
\[
a(G) = \max \{ k \in Z | \sum_{i=1}^k d_i \leq m(G) \}.
\]
\end{defn}

Since the sum of the first $\lfloor \frac{n}{2} \rfloor$ terms in $D$ is clearly at most $m(G)$, it is apparent from the definition above that $a(G) \geq \lfloor \frac{n}{2} \rfloor$.

\begin{thm}\label{alpha_ann}\cite{3,4,5}
For any graph $G$, $\alpha(G) \leq a(G)$.
\end{thm}

To see that the definition and theorem above are a special case of the DSI strategy, notice that we would just make the identifications, $f_U(G,A(P))=0$ and $DSI_U(G,h,f_U,c)=a(G)$, while letting $P$ be the property of being an independent set.

Our first new application of the DSI strategy is a generalization and improvement of Theorem \ref{alpha_ann}, as well as a new and analogous lower bound.

\begin{defn}\label{upper_lower}
Let $D=\{d_1 \leq d_2 \leq \ldots \leq d_n\}$ be the degree sequence of a graph $G=(V,E)$ and let $F$ denote the family of all maximum $j$-independent sets in $G$.  The \textit{upper $j$-annihilation number} of $G$, denoted $a_j=a_j(G)$, can be defined by the equation: 
\[
a_j(G) = \max \{ k \in Z | \sum_{i=1}^k d_i + \max_{S \in F} \{m[V-S]-m[S] \} \leq m(G) \}.
\]
The \textit{lower $j$-annihilation number} of $G$, denoted $c_j=c_j(G)$, can be defined by the equation: 
\[
c_j(G) = \min \{ k \in Z | \sum_{i=1}^k d_{n-i+1} + \min_{S \in F} \{m[V-S]-m[S] \} \geq m(G) \}.
\]
\end{defn}

The main result of this section now follows.

\begin{thm}\label{main}
For any positive integer $j$ and for any graph $G=(V,E)$; 
\[
c_j(G) \leq \alpha_j(G) \leq a_j(G).
\]
\end{thm}

\begin{proof}
First we prove the upper bound.  Let $I$ be a maximum $j$-independent set such that for all $S \in F$, $m[V-I]-m[I] \geq m[V-S]-m[S]$. Denote by $m_1$ the number of edges in $[I]$, by $m_2$ the number of edges in $[V-I]$, and by $m_3$ the number of edges between $I$ and $V-I$. Observe the following chain of inequalities:
\[
\sum_{i=1}^{\alpha_j} d_i + (m[V-I]-m[I]) \leq \sum_{v \in I} deg(v) + m_2-m_1 = 2m_1+m_3 + m_2-m_1 = m.
\]
Since $\alpha_j$ is an integer satisfying the condition in the definition of the upper $j$-annihilation number, and $a_j$ is the largest such integer, the upper bound is proven.

Next we prove the lower bound.  Let $I$ be a maximum $j$-independent set such that for all $S \in F$, $m[V-I]-m[I] \leq m[V-S]-m[S]$. Denote $m_1$, $m_2$, and $m_3$ as above. Observe the following chain of inequalities:
\[
\sum_{i=1}^{\alpha_j} d_{n-i+1} + (m[V-I]-m[I]) \geq \sum_{v \in I} deg(v) + m_2-m_1 = 2m_1+m_3 + m_2-m_1 = m.
\]
Since $\alpha_j$ is an integer satisfying the condition in the definition of the lower $j$-annihilation number, and $c_j$ is the smallest such integer, the lower bound is proven.
\end{proof}

To see that these results fit into the DSI strategy, note that our property $P$ is that of being a $j$-independent set and $A(P)$ is a maximum $j$-independent set (so that $|A(P)|=\alpha_j(G)$).  Moreover, our functions $f_U(G,A(P))$ and $f_L(G,A(P))$ are $\max \{m[V-S]-m[S]| S \in F \}$ and $\min \{m[V-S]-m[S]| S \in F \}$ respectively. Finally, $DSI_U(G,h,f_U,c)$ and $DSI_L(G,h,f_L,c)$ are simply the upper and lower $j$-annihilation numbers.

The quality of Theorem \ref{main} will now be discussed. First, let us consider a few examples where the upper $j$-annihilation number is an improvement on the annihilation number from Definition \ref{ann_def}.

\begin{ex}[Showing $\alpha(G)=a_1(G)<a(G)$]\label{3.5}
For positive integers $p$ and $n$ satisfying $2p+3 < n$, the families of graphs $E_p \cup K_{n-p}$ and $E_p + K_{n-p}$ are both examples where $\alpha(G)=a_1(G)<a(G)$.  In fact, we have;
\[
\alpha(E_p \cup K_{n-p})=a_1(E_p \cup K_{n-p})= p+1 < \frac{n-1}{2} \leq \lfloor \frac{n}{2} \rfloor \leq a(E_p \cup K_{n-p})
\]
and
\[
\alpha(E_p + K_{n-p})=a_1(E_p + K_{n-p})= p < \frac{n-3}{2} < \lfloor \frac{n}{2} \rfloor \leq a(E_p + K_{n-p}).
\]
\end{ex}

\begin{ex}[Showing $\alpha(G)=2=a_1(G)<a(G)$]\label{3.6}
Let $G$ be the graph obtained by adding a matching between two complete graphs with $p$ vertices. Then, if $ p \geq3$, we have;
\[
\alpha(G)=a_1(G)= 2 < p = a(G).
\]
\end{ex}

\begin{rem}
It should be mentioned here that, while the upper $j$-annihilation number is sharp for every $p \geq3$ in the graphs described in Example \ref{3.6}, none of the known upper bounds on the independence number presented in the recent survey \cite{9} are satisfied with equality for these examples.  This includes some of the more famous bounds such as $\alpha(G) \leq n(G)-\mu(G)$, as well as the bound of Cvetkovic of the minimum of the non-negative and non-positive eigenvalues of the adjacency matrix.  Thus there are examples where this new bound, the upper $1$-annihilation number is a better approximation of $\alpha_1=\alpha$ than all known upper bounds. Additionally, if $G$ is the graph of the regular dodecahedron, we have $c_1(G) = 8 = \alpha(G)$, while all of the 12 lower bounds on the independence number presented in \cite{9} return values less than that.  Hence there are also examples where the lower $1$-annihilation number is a better approximation of $\alpha_1=\alpha$ than all known lower bounds. Admittedly, in both instances, focus was on efficiently computable approximations and neither $a_j$ nor $c_j$ have this property.
\end{rem}

Next we present a theorem that shows the strengths and weaknesses of Theorem \ref{main} in its most general form.  In particular, it will show that there are graphs where equality holds throughout the theorem, while also graphs where both upper and lower bounds can be very far from the actual value of $\alpha_j(G)$. The fact that both upper and lower bounds can, for some graphs, be very poor approximations to the independence number (the $j=1$ case) is not surprising when one considers that determining $\alpha(G)$ is a well known NP-hard problem \cite{15,14,13}.  In this context, the following theorem gives more evidence that, in spite of the apparent improvement over known upper and lower bounds, the situation is still far from ideal.

\begin{thm} We will give constructive existence proofs of the following four propositions.
\begin{enumerate}
  \item There exists graphs where $c_j(G) = \alpha_j(G) = a_j(G)$.
  \item There exists graphs where $c_j(G) = \alpha_j(G)$ while $\frac{a_j(G)}{\alpha_j(G)} \to \infty$.
  \item There exists graphs where $a_j(G) = \alpha_j(G)$ while $\frac{\alpha_j(G)}{c_j(G)} \to \infty$.
  \item There exists graphs where $\alpha_j(G)-c_j(G) \to \infty$ and $\frac{a_j(G)}{\alpha_j(G)} \to \infty$.
\end{enumerate}
\end{thm}

%\begin{ex}[Showing $c_j(G) = \alpha_j(G) = a_j(G)$]
\begin{proof}

To prove (1), let $G$ be a regular graph whose vertices can be partitioned into two maximum $j$-independent sets. As evidence that these kind of graphs exist in general, let
\[ G = (\cup_{i=1}^{j}K_j) + (\cup_{i=1}^{j}K_j) \].
Notice that this family of graphs is regular of degree $j^2+j-1$, has $\alpha_j(G)=j^2$, and its vertices can be partitioned into two maximum $j$-independent sets. First, note that $a_j \leq \frac{n}{2}$ since the sum of the smallest $\frac{n}{2}$ terms of the degree sequence is already equal to $m(G)$, and $\max \{m[V-S]-m[S]| S \in F \} \geq 0$ since one of the two parts has at least as many edges in its induced subgraph as the other. Next, note that $c_j(G) \geq \frac{n}{2}$ since the sum of the largest $\frac{n}{2}$ terms of the degree sequence is already equal to $m(G)$, and $\min \{m[V-S]-m[S]| S \in F \} \leq 0$ since one of the two parts has at most as many edges in its induced subgraph as the other. Combining this with Theorem \ref{main},
\[
\frac{n(G)}{2} \leq c_j(G) \leq \alpha_j(G) \leq a_j(G) \leq \frac{n(G)}{2}.
\]
Thus they are all equal and the first proposition is established.

To prove (2), we denote the disjoint union of $b$ isomorphic copies of $H$ with the notation $\cup_{i=1}^b H$. Now, for given positive integers $j$ and $p$ such that $j < p^2$, we will establish the truth of the following two claims;
\[
c_j(\cup_{i=1}^{p}K_{p^2} + \cup_{i=1}^{p+1}K_j)= \alpha_j(\cup_{i=1}^{p}K_{p^2} + \cup_{i=1}^{p+1}K_j),
\]
while simultaneously, as $p \to \infty$;
\[
\frac{a_j(\cup_{i=1}^{p}K_{p^2} + \cup_{i=1}^{p+1}K_j)}{\alpha_j(\cup_{i=1}^{p}K_{p^2} + \cup_{i=1}^{p+1}K_j)} \to \infty.
\]

For ease of notation, set $G=\cup_{i=1}^{p}K_{p^2} + \cup_{i=1}^{p+1}K_j$.  First, we can see that $\alpha_j(G)=j(p+1)$. Observe that $G$ has $p^3$ vertices of degree $p^2+pj+j-1$ while it has $j(p+1)$ vertices of degree $p^3+j-1$.  When combined with the fact that $G$ has a unique maximum $j$-independent set -- the $(p+1)$ copies of $K_j$ -- this allows us to deduce all of the following;
\[
\min_{S \in F} \{m[V-S]-m[S] \} = \frac{p^3(p^2-1)}{2} - \frac{j(j-1)(p+1)}{2},
\]
\[
m(G) = \frac{p^3(p^2-1)}{2} + \frac{j(j-1)(p+1)}{2} + p^3(p+1)j,
\]
\[
\sum_{i=1}^{(p+1)j} d_{n-i+1} + \min_{S \in F} \{m[V-S]-m[S] \} = (p^3+j-1)(p+1)j + \frac{p^3(p^2-1)}{2} - \frac{j(j-1)(p+1)}{2} = m(G).
\]
Therefore, we conclude from the definition, $c_j(G)=\alpha_j(G)$, which settles the first claim.

As for the second claim, since $\alpha_j(G)=j(p+1)$, it only remains to calculate $a_j(G)$ and compare them. To this end, we observe the validity of the following chain of inequalities, as $p \to \infty$;
\[
\sum_{i=1}^{p^2} d_i + \max_{S \in F} \{m[V-S]-m[S] \} = (p^2+pj+j-1)p^2 + \frac{p^3(p^2-1)}{2} - \frac{j(j-1)(p+1)}{2} \leq m(G).
\]
This shows that $p^2$ is an integer satisfying the condition in the definition of upper $j$-annihilation number.  Hence, because $a_j(G)$ is the largest such integer, $a_j(G)\geq p^2$.  Finally, 
\[
\frac{a_j(G)}{\alpha_j(G)} \geq \frac{p^2}{j(p+1)},
\]
and the right hand side of this inequality grows arbitrarily large with $p$ for any fixed integer $j$. This completes the proof of the second proposition.

Next, to prove (3), consider complete split graphs, which are joins of complete graphs and empty graphs.  Note that these graphs are the same as those from Example \ref{3.5}. Let $p$ and $j$ be positive integers.  Then, we make the following two claims.  As $p \to \infty$, we have;
\[
\alpha_j(E_{p^2}+K_p)= a_j(E_{p^2}+K_p),
\]
while simultaneously, 
\[
\frac{\alpha_j(E_{p^2}+K_p)}{c_j(E_{p^2}+K_p)} \to \infty.
\]

For ease of notation, set $G=E_{p^2}+K_p$.  Choose $p$ large enough so that, $j < p^2$. It is clear now that $\alpha_j(G)=p^2$.  Observe that $G$ has $p^2$ vertices of degree $p$ while it has $p$ vertices of degree $p^2+p-1$.  Now, since the difference $m[V-S]-m[S]$ is maximized over the set $F$, of all maximum $j$-independent sets, when $S$ is the vertex set of $E_{p^2}$, we see that;
\[
\max_{S \in F} \{m[V-S]-m[S] \} = \frac{p(p-1)}{2},
\]
while,
\[
m(G)=\frac{p(p-1)}{2}+p^3.
\]
Therefore,
\[
\sum_{i=1}^{p^2} d_i + \max_{S \in F} \{m[V-S]-m[S] \} = p^3 + \frac{p(p-1)}{2} = m(G),
\]
from which we conclude from the definition that $a_j(G) = p^2$, which settles the first claim.

As for second claim, since we still have $\alpha_j(G)=p^2$, it only remains to calculate $c_j(G)$ and compare them.  To this end, we choose $p$ large enough so that for the given positive integer $j$, $G$ has a unique maximum $j$-independent set consisting of the vertex set of $E_{p^2}$. Then we see that;
\[
\min_{S \in F} \{m[V-S]-m[S] \} = \frac{p(p-1)}{2}.
\]
Therefore,
\[
\sum_{i=1}^{p} d_{n-i+1} + \min_{S \in F} \{m[V-S]-m[S] \} = p^3 + \frac{p(p-1)}{2} = m(G).
\]
We now conclude from the definition that $c_j(G)=p$.  Hence $\frac{\alpha_j(G)}{c_j(G)}=p \to \infty$, concluding the proof of the third proposition.

Finally, to prove (4), we establish the truths of the two claims that follow.  For positive integers $p$,$q$,$r$, and $j$, define the graph $G(p,q,r,j)$ as follows.  Starting with the disjoint union of a $K_q$ and $q$ disjoint copies of the graph $\cup_{i=1}^{p}K_{r} + \cup_{i=1}^{p+1}K_j$, associate a unique vertex of $K_q$ to each of the $q$ copies and join all vertices of each copy to this associated vertex. Now, as $p \to \infty$, we must show that;
\[
\alpha_j(G(p,p^2,p^2,j)) - c_j(G(p,p^2,p^2,j)) \to \infty,
\]
while simultaneously, 
\[
\frac{a_j(G(p,p^2,p^2,j))}{\alpha_j(G(p,p^2,p^2,j))} \to \infty.
\]

Set $G=G(p,p^2,p^2,j)$ for ease of notation.  Since we are only interested in this family of graphs when $p$ is growing arbitrarily large, we only calculate the following list of invariants for $p$ large enough so that $j < p^2$. First, since $G$ has a unique maximum $j$-independent set, we can record the following invariants;
\[
\alpha_j(G)= p^3j + p^2j,
\]
\[
m(G) = (\frac{p^3(p^2-1)}{2} + \frac{j(j-1)(p+1)}{2}+ p^3(p+1)j)p^2 + \frac{p^2(p^2-1)}{2} + p^2(p^3+j(p+1)),
\]
\[
\min_{S \in F} \{m[V-S]-m[S] \}=\max_{S \in F} \{m[V-S]-m[S] \} = \frac{p^5(p^2-1)}{2}+\frac{p^2(p^2-1)}{2} + p^5 - \frac{p^2(p+1)j(j-1)}{2}.
\]
Subtracting the third equation above from the second and simplifying, we have;
\begin{equation}\label{p6}
m(G) - \min_{S \in F} \{m[V-S]-m[S] \}=m(G) - \max_{S \in F} \{m[V-S]-m[S] \} = p^6j + p^5j +p^3j^2 + p^2j^2.
\end{equation}

Now we will record the degree sequence of $G$, using exponents on the different degrees to indicate the number of times that degree occurs in $G$,
\[
D = \{(p^2+pj+j)^{p^5},(p^3+j)^{(p+1)p^2j},(p^3+p^2+pj+j-1)^{p^2} \}.
\]
From this we observe that the sum of the largest $p^2$ degrees is less than Equation \ref{p6}, while the sum of the largest $(p^2+(p+1)p^2j)$ degrees is greater than Equation \ref{p6}.  This enables us to deduce that $p^2 < c_j(G) \leq p^2+(p+1)p^2j$.  Using this information, we derive the following;
\[
\sum_{i=1}^{p^3j+p^2j-p} d_{n-i+1} = p^2(p^3+p^2+pj+j-1) + (p^3j+p^2j-p-p^2)(p^3+j) \geq m(G) - \min_{S \in F} \{m[V-S]-m[S] \}.
\]
This shows that $(p^3j+p^2j-p)$ is an integer satisfying the condition in the definition of lower $j$-annihilation number. Since $c_j(G)$ is defined as the smallest such integer, $c_j(G) \leq p^3j+p^2j-p$.  Hence we can compare the difference between $\alpha_j(G)$ and $c_j(G)$ as follows;
\[
\alpha_j(G)-c_j(G) \geq (p^3j+p^2j) - (p^3j+p^2j-p) = p,
\]
which can be made arbitrarily large.  This completes the proof of the first claim.

Next, we observe that the smallest $p^5$ degrees of $G$ are all $(p^2+pj+j)$. Of course, the same thing is true for the smallest $p^4-p^2$ degrees, from which we get;
\[
\sum_{i=1}^{p^4-p^2} d_i = (p^4-p^2)(p^2+pj+j) \leq m(G) - \max_{S \in F} \{m[V-S]-m[S] \}.
\]
This shows that $p^4-p^2$ is an integer satisfying the condition in the definition of upper $j$-annihilation number. Since $a_j(G)$ is defined as the largest such integer, $a_j(G) \geq p^4-p^2$.  Hence, we can compare the ratio of $\alpha_j(G)$ and $a_j(G)$ as follows;
\[
\frac{a_j(G)}{\alpha_j(G)} \geq \frac{p^4-p^2}{p^3j + p^2j}, 
\]
which can be made arbitrarily large.  This establishes the second claim, completes the proof of the fourth proposition, and therefore proves the theorem.
\end{proof}

To conclude this section, we give a couple more definitions and a lemma that will be used later on in the paper. Recall the definitions of upper and lower $j$-annihilation number, where $F$ is the family of all maximum $j$-independent sets,
\[ a_j(G) = \max \{ k \in Z | \sum_{i=1}^k d_i + \max_{S \in F} \{m[V-S]-m[S] \} \leq m(G) \}, \]
\[ c_j(G) = \min \{ k \in Z | \sum_{i=1}^k d_{n-i+1} + \min_{S \in F} \{m[V-S]-m[S] \} \geq m(G) \}. \]
In \cite{2}, the authors define an annihilating set to be a set whose degree sum is at most the size.  We will borrow this language to define an \textit{upper $j$-annihilating set} to be a set $A$ with the property that,
\[ \sum_{v \in A}deg(v) + \max_{S \in F} \{m[V-S]-m[S] \} \leq m(G). \]
Analogously, we say that $A$ is a \textit{lower $j$-annihilating set} when,
\[ \sum_{v \in A}deg(v) + \min_{S \in F} \{m[V-S]-m[S] \} \geq m(G). \]
We then define a \textit{maximum upper $j$-annihilating set} to be an upper $j$-annihilating set of the largest order and a 
\textit{minimum lower $j$-annihilating set} is a lower $j$-annihilating set of the smallest order.

\begin{lem}\label{ann_set}
For any graph $G$, if $A$ is a maximum upper $j$-annihilating set, then $|A|=a_j(G)$.  That is, the order of a maximum upper $j$-annihilating set is exactly the upper $j$-annihilation number. Moreover, if $A$ is a minimum lower $j$-annihilating set, then $|A|=c_j(G)$.  That is, the order of a minimum lower $j$-annihilating set is exactly the lower $j$-annihilation number.
\end{lem}

\begin{proof}
Let $G$ be a graph with degree sequence $D=\{d_1 \leq d_2 \leq \ldots \leq d_n\}$ and let $A$ be a maximum upper $j$-annihilating set of $G$.  Now,
\[ \sum_{i=1}^{|A|}d_i \leq \sum_{v \in A} deg(v), \]
which implies
\[ \sum_{i=1}^{|A|}d_i + \max_{S \in F} \{m[V-S]-m[S] \} \leq \sum_{v \in A} deg(v) + \max_{S \in F} \{m[V-S]-m[S] \} \leq m(G). \]
Hence, $|A|$ is an integer satisfying the definition of upper $j$-annihilation while $a_j(G)$ is the largest such integer.  Consequently, $|A| \leq a_j(G)$.

On the other hand, let $B$ be a set of vertices of $G$ whose degrees are the $a_j(G)$ smallest degrees, $\{d_1,\ldots,d_{a_j}\}$. Clearly we have,
\[ \sum_{i=1}^{a_j}d_i = \sum_{v \in B} deg(v). \]
Hence, from the definition of $a_j(G)$,
\[ \sum_{v \in B} deg(v) + \max_{S \in F} \{m[V-S]-m[S] \} = \sum_{i=1}^{a_j}d_i + \max_{S \in F} \{m[V-S]-m[S] \} \leq m(G). \]
From this we conclude that $B$ is an upper $j$-annihilating set and as such, the order of $B$ is less than or equal to the order of a maximum upper $j$-annihilating set -- namely, $|B| \leq |A|$.  Therefore, since $a_j(G)=|B|$ and together with the first paragraph, this shows they are equal and proves the first part of the theorem. The second part of the theorem can be proven in a similar fashion.

\end{proof}

Of course, there could be more than one maximum upper $j$-annihilating set, but the proof shows that they all have the same order and additionally that any set of the smallest $a_j(G)$ degrees suffices to find one.

It is clear that calculating $a_j(G)$ and $c_j(G)$ is still an intractable problem, since it uses information about all maximum $j$-independent sets.  Thus, as was alluded to in the description of the DSI strategy, the next step is to find approximations to the functions $f_U(G,A(P))$ and $f_L(G,A(P))$, so that the weaker bounds can at least be computed more easily. This is done to some extent in the next section, where we also give applications of the DSI strategy when certain other features of the graph are known.

\section{Approximations and Applications}

In  this section, we first give easily computable approximations to Theorem \ref{main}.  These are presented in Definition \ref{weak_upper_lower} and Theorem \ref{weak}.  After that, we illustrate what can be gained by assuming the graph is planar.  Next, we give an application using chromatic number.  Finally, we apply the DSI strategy to claw-free graphs. 

Let us recall once again Definition \ref{upper_lower}, the upper and lower $j$-annihilation numbers of $G$;
\[
\alpha_j(G) \leq a_j(G) = \max \{ k \in Z | \sum_{i=1}^k d_i + \max_{S \in F} \{m[V-S]-m[S] \} \leq m(G) \},
\]
\[
\alpha_j(G) \geq c_j(G) = \min \{ k \in Z | \sum_{i=1}^k d_{n-i+1} + \min_{S \in F} \{m[V-S]-m[S] \} \geq m(G) \}.
\]
Our next step in the DSI strategy is to find approximations to $\max_{S \in F} \{m[V-S]-m[S] \}$ and $\min_{S \in F} \{m[V-S]-m[S] \}$ that are simpler or at least more easily computed.  In particular, we need a simpler function $f(S) \leq \max_{S \in F} \{m[V-S]-m[S] \}$, such that when substituted into the definition, we get an index at least as large as $a_j(G)$.  For the lower bound, we need to find a simpler function $g(S) \geq \min_{S \in F} \{m[V-S]-m[S] \}$, such that when substituted into the definition, we get an index at most as large as $c_j(G)$. To illustrate this idea with an example, consider the following definitions, which give easy to calculate approximations for the invariants introduced in Theorem \ref{main}.

\begin{defn}\label{weak_upper_lower}
Let $D=\{d_1 \leq d_2 \leq \ldots \leq d_n\}$ be the degree sequence of a graph $G=(V,E)$. The \textit{weak upper $j$-annihilation number} of $G$, denoted $a_j^\prime=a_j^\prime(G)$, can be defined by the equation: 
\[
a_j^\prime(G) = \max \{ k \in Z | \sum_{i=1}^k d_i - \frac{k(j-1)}{2} \leq m(G) \}.
\]
The \textit{weak lower $j$-annihilation number} of $G$, denoted $c_j^\prime=c_j^\prime(G)$, can be defined by the equation: 
\[
c_j^\prime(G) = \min \{ k \in Z | \sum_{i=1}^k d_{n-i+1} + \frac{1}{2}\sum_{i=1}^{n-k} (d_{n-i+1}-1) \geq m(G) \}.
\]
\end{defn}

\begin{thm}\label{weak}
For any positive integer $j$ and for any graph $G=(V,E)$; 
\[
c_j^\prime(G) \leq c_j(G) \leq \alpha_j(G) \leq a_j(G) \leq a_j^\prime(G).
\]
\end{thm}

\begin{proof}
To see that $c_j^\prime(G) \leq c_j(G)$, it is enough to show;
\[
\frac{1}{2}\sum_{i=1}^{n-c_j} (d_{n-i+1}-1) \geq \min_{S \in F} \{m[V-S]-m[S] \}.
\]
With this in mind, let $A$ be a maximum $j$-independent set which realizes $\min_{S \in F} \{m[V-S]-m[S] \}$. Denote by $m_1$ the number of edges in $[A]$, by $m_2$ the number of edges in $[V-A]$, and by $m_3$ the number of edges between $A$ and $V-A$. Now we get the following three equations, which will simplify what follows;
\[
m[V-A]-m[A] = m_2 - m_1,
\]
\[
\sum_{v \notin A} deg(v) = 2m_2 + m_3.
\]
\[
\sum_{v \in A} deg(v) = 2m_1 + m_3.
\]
From Theorem \ref{main}, $c_j(G) \leq \alpha_j(G)$, so;
\[
\frac{1}{2}\sum_{i=1}^{n-c_j} (d_{n-i+1}-1) \geq \frac{1}{2}\sum_{i=1}^{n-\alpha_j} (d_{n-i+1}-1) = \frac{1}{2}\sum_{i=1}^{n-\alpha_j} d_{n-i+1} - \frac{n-\alpha_j}{2}.
\]
However, since the sum of the highest $n-\alpha_j$ degrees is at least as large as the sum of the degrees of the $n-\alpha_j$ vertices in $V-A$,
\[
\frac{1}{2}\sum_{i=1}^{n-\alpha_j} d_{n-i+1} - \frac{n-\alpha_j}{2} \geq \frac{1}{2}\sum_{v \notin A} deg(v) - \frac{n-\alpha_j}{2} = m_2 + \frac{1}{2}m_3 - \frac{n-\alpha_j}{2}.
\]
Next we observe that,
\[
m_2 + \frac{1}{2}m_3 - \frac{n-\alpha_j}{2} \geq \min_{S \in F} \{m[V-S]-m[S] = m[V-A]-m[A] = m_2 - m_1,
\]
if and only if,
\[
n-\alpha_j \leq 2m_1 + m_3 = \sum_{v \in A} deg(v).
\]
But this last equation is true since each of the $n-\alpha_j$ vertices not in $A$ has a neighbor in $A$ due to the fact that $A$ is a maximum $j$-independent set. From this we conclude that $c_j^\prime(G) \leq c_j(G) \leq \alpha_j(G)$.

On the other hand, to see that $a_j^\prime(G) \geq a_j(G)$, let $I$ be a maximum $j$-independent set such that for all $S \in F$, $m[V-I]-m[I] \geq m[V-S]-m[S]$. Now, since $\alpha_j \leq a_j$, $m[V-I] \geq 0$, and $m[I] \leq \frac{\alpha_j(j-1)}{2}$, we deduce;

\[
\sum_{i=1}^{a_j} d_i - \frac{a_j(j-1)}{2} \leq \sum_{i=1}^{a_j} d_i - \frac{\alpha_j(j-1)}{2} \leq \sum_{i=1}^{a_j} d_i + m[V-I] - m[I] \leq m(G).
\]

As $a_j^\prime(G)$ is the largest integer having this property by definition, $a_j(G) \leq a_j^\prime(G)$, completing the proof.
\end{proof}

\begin{rem}
We should note here that the weak annihilation number of $G$, $a_j^\prime(G)$, is exactly equal to the annihilation number from Definition \ref{ann_def} when $j=1$. Thus Theorem \ref{weak} is a generalization of Theorem \ref{alpha_ann}, while Theorem \ref{main} is both a generalization and an improvement. The definition of $a_j^\prime(G)$, and its relationship to $\alpha_j(G)$, was previously discovered by Pepper and Waller \cite{12}, though it was never published.
\end{rem}

The main idea of this section was to make use of the DSI strategy to create efficient approximations for NP-hard invariants that we are interested in.  As more information about the graph is known, the approximations can be made more precise.  In fact, the definition of the weak lower $j$-annihilation number does not even depend on $j$, so that $c_j^\prime(G)=c_1^\prime(G)$. %This problem will be partly addressed in a later section, to give a better approximate lower bound for $\alpha_j(G)$, but which is related to an application of DSI to $j$-domination number. 
Moreover, the weak upper $j$-annihilation number was defined without any consideration for the edges outside of a maximum $j$-independent set, even though this was part of the definition for the upper $j$-annihilation number. Some of these weaknesses can be addressed by knowing more about the structure of the graph.

\subsection{Approximations assuming planarity}

With that in mind, let us turn our attention to maximum planar graphs, that is, planar graphs $G$ such that $m(G)=3n(G)-6$.

\begin{thm}\label{planar}
Let $G$ be a maximum planar graph with minimum degree $\delta(G) \leq 5$.  Then, for any positive integer $j \leq \delta(G)$;
\[
\alpha_j(G) \leq \frac{2n(G)-4}{\delta(G)-j+1}.
\]
\end{thm}

\begin{proof}
For any maximum independent set $S$ in $G$, the number of edges between $S$ and $V-S$ is at most $2n(G)-4$, since that is true for all bipartite planar graphs by Euler's formula.  Now, because $G$ is maximum planar, $m(G) =3n(G)-6$.  Hence we have that;
\[
\max_{S \in F} \{m[V-S]-m[S] \} \geq \max_{S \in F} \{m[V-S]\} \geq (3n(G)-6)-(2n(G)-4)=n(G)-2,
\] 
where $F$ is the set of all maximum $j$-independent sets. So this gives us the approximation we need to more precisely apply the DSI strategy, in the form of the upper $j$-annihilation number.  Now, using the above inequality together with the fact that $\alpha_j(G)\leq a_j(G)$ from Theorem \ref{main}, we get;
\[
(\delta(G)-j+1) \alpha_j(G) + n(G)-2 \leq \sum_{i=1}^{\alpha_j} d_i + n(G)-2 \leq \sum_{i=1}^{a_j} d_i + \max_{S \in F} \{m[V-S]-m[S] \} \leq m(G) = 3n(G)-6.
\] 
From which we deduce our desired inequality.
\end{proof}

When $j=1$, this result becomes a theorem from a paper of Caro and Roddity \cite{10}. In that paper, examples are given showing that equality holds in Theorem \ref{planar}, when $j=1$, for infinitely many graphs and for each value of $\delta \in \{2,3,4,5\}$.  A nice corollary to Theorem \ref{planar} that is worth mentioning is below.

\begin{cor}\cite{10}
If $G$ is a maximum planar graph with $\delta(G)=5$, then;
\[
\alpha(G) \leq \frac{2n(G)-4}{5}.
\]
\end{cor}

To see that the inequality in Theorem \ref{planar} is sharp even for $j>1$, consider the following example.  We show here that the theorem is satisfied with equality (only asymptotically in one case) for the following pairs $(j,\delta)$: $\{(1,3),(1,4),(1,5),(2,2),(2,3),(2,4),(3,4),(3,5)\}$.  For the pairs $(3,3),(4,4),$ and $(5,5)$ equality is not possible. The cases when $j=1$ appeared in \cite{10}, though we give another example of the $(1,4)$ case below.  Examples for the other cases where equality holds are collected below.  We do not know whether the bound is sharp for some graphs for the $(j,\delta)=(2,5)$ and $(j,\delta)=(4,5)$ cases, and leave these for open problems.

\begin{ex}
First, when $j=\delta(G)=2$, a complete graph on $3$ vertices is the unique graph with the desired properties, and there are no other instances where equality can be achieved when $j = \delta(G)$.  So we restrict our attention to $j \leq \delta(G)-1$.  When $j=2$ and $\delta(G)=3$, the complete graph on $4$ vertices is the unique graph with the desired properties.
 
Now, let $G$ be the graph formed by taking a cycle on $3p$ vertices, where $p \geq 2$ is an integer,  adding a vertex $u$ inside this cycle and a vertex $v$ outside the cycle, and then joining each of $u$ and $v$ to each of the $3p$ vertices of the cycle. Observe that $m(G) = 9p = 3n(G)-6$.  Hence, $G$ is a maximum planar graph with $\delta(G)=4$.  When $j=1$, we find that,
\[ \frac{2n(G)-4}{\delta(G)-j+1} =  \frac{2n(G)-4}{4}= \frac{n(G)}{2} - 1 = \alpha_1(G) = \alpha(G). \]
Moreover, when $j=2$, we find,
\[ \frac{2n(G)-4}{\delta(G)-j+1} =  \frac{2n(G)-4}{3}= 2p = \alpha_2(G). \]
Finally, when $j=3$, we find,
\[ \frac{2n(G)-4}{\delta(G)-j+1} = n(G)-2 = 3p = \alpha_3(G). \]
Thus, we see there are infinitely many examples satisfying Theorem \ref{planar} with equality when $\delta =4$ and $1 \leq j \leq 3$.

When the $\delta(G)=5$ and $j=3$, we can construct a family of graphs for which the inequality is ``nearly" sharp, meaning different only by a small constant as $n$ grows arbitrarily large.  Let $G$ be the graph described as follows.  Let $r \geq 5$ be an integer.  Let $A=P_r$ be a path on $r$ vertices labeled $\{a_1,\ldots,a_r\}$, let $B=P_r$ be a path on $r$ vertices labeled $\{b_1,\ldots,b_r\}$, and let $C=P_{r-1}$ be a path on $r-1$ vertices labeled $\{c_1,\ldots,c_{r-1}\}$.  Draw $A$ above $C$ above $B$.  Join $a_1$ to $b_1$ and join $a_r$ to $b_r$.  For each $i \in \{1,\ldots,r-1\}$, join $c_i$ to $a_i$, $a_{i+1}$, $b_i$, and $b_{i+1}$.  Add a vertex $u$, and join it to all of the vertices of $A$, and a vertex $v$, and join it to all of the vertices of $B$.  Finally, join $a_1$ to $a_r$, join $b_1$ to $b_r$, and $a_1$ to $b_r$.  This graph is maximum planar with $\delta(G)=5$.  When $j=3$, the set $(A-a_1)\cup(B-b_r)$ is a $3$-independent set of order $2r-2$. Moreover, since $n=3r+1$
\[ \frac{2n(G)-4}{\delta(G)-j+1} =  \frac{2n(G)-4}{3} = 2r - \frac{2}{3} = \alpha_3(G) + \frac{4}{3}. \]
\end{ex}

\subsection{Approximations using chromatic number}

Next we will focus on using the chromatic number to get upper bounds on the independence number.  Before proceeding, let us define the \textit{$j$-chromatic number of $G$}, denoted by $\chi_j(G)$, as the fewest number of $j$-independent sets the vertices of $G$ can be partitioned into.  For example, when $j=1$, this is just the regular chromatic number.

\begin{thm}\label{chromatic}
For any positive integer $j$ and for any graph $G=(V,E)$;
\[
\alpha_j(G) \leq a_j(G) \leq \max \{ k \in Z | \sum_{i=1}^k d_i  + {\chi_j-1 \choose 2}  -\frac{k(j-1)}{2} \leq m(G) \} \leq a_j^\prime(G).
\]
\end{thm}

\begin{proof}
To prove this result, as with the earlier approximations of Theorem \ref{main}, it suffices to establish that;
\[
{\chi_j-1 \choose 2} - \frac{a_j(G)(j-1)}{2} \leq \max_{S \in F} \{m[V-S]-m[S] \},
\]
where again, $F$ is the family of all maximum $j$-independent sets.

To this end, let $A$ be a maximum $j$-independent set realizing $\max_{S \in F} \{m[V-S]-m[S] \}$.  The first point to observe is that;
\[ 
\chi_j([V-A]) \geq \chi_j(G)-1,
\]
since otherwise, because $A$ is $j$-independent, we could have found a smaller partition than $\chi_j(G)$. Now, partition $V-A$ into $\chi_j([V-A])$ $j$-independent sets.  The next point to observe is that there is at least one edge between all pairs of these sets, due to the minimality of the coloring. Consequently;
\[
m[V-A] \geq {\chi_j-1 \choose 2}.
\]  
To conclude, since we know that each vertex of $A$ is adjacent to at most $j-1$ others, and $\alpha_j(G) \leq a_j(G)$; 
\[
m[A] \leq \frac{\alpha_j(G)(j-1)}{2} \leq \frac{a_j(G)(j-1)}{2}
\]
\end{proof}

For the $j=1$ case, we have the following corollary.

\begin{cor}
For any graph $G=(V,E)$;
\[
\alpha(G) \leq \max \{ k \in Z | \sum_{i=1}^k d_i  + {\chi-1 \choose 2} \leq m(G) \}.
\]
\end{cor}
With this corollary in hand for instance, we could get a slight improvement over the naive annihilation number upper bound for all planar graphs, where we know that $\chi(G) \leq 4$ from the Four Color Theorem.  It also gives us another way to interpret the intuitive idea that highly chromatic graphs have relatively small maximum independent sets, since high chromatic number would generally push the upper bound above lower.
%\subsection{Specification to 1-Independence}

%In this section, we will focus on the $j=1$ case of the theorems above -- namely the application of the upper and lower annihilation number to the independence number.

\subsection{Specification to Claw-Free and $K_{1,p}$-Free Graphs}

Now we focus specifically on using the DSI strategy to find approximations to the independence number for $K_{1,p}$-free graphs (graphs which have no induced $K_{1,p}$).

\begin{thm}\label{claw}
Let $p\geq 3$ be an integer and let $G=(V,E)$ be a $K_{1,p}$-free graph with degree sequence $D=\{d_1 \leq d_2 \leq \ldots \leq d_n\}$.  Then,
\[ a_1(G) \leq w(G) = \max \{ k \in Z | \sum_{i=1}^k d_i + \frac{1}{2}\sum_{i=k+1}^n d_i - \frac{(n-k)(p-1)}{2} \leq m(G) \}. \]
\end{thm}

\begin{proof}
Let $A$ be a maximum upper $1$-annihilating set of $G$. First we observe that,
\[ \sum_{i=1}^{|A|}d_i + \frac{1}{2}\sum_{i=|A|+1}^n d_i - \frac{(n-|A|)(p-1)}{2} \leq \sum_{v\in A} deg(v) + \frac{1}{2}\sum_{v \notin A} deg(v) - \frac{(n-|A|)(p-1)}{2}. \]
This is true because the weight of the first sum is $1$ while that of the second is $\frac{1}{2}$, so any deviation from the lowest $|A|$ terms of $D$ being the degrees of the vertices in $A$ would only favor the right hand side of the above inequality.  Next we observe that,
\[ \sum_{v\in A} deg(v) + \frac{1}{2}\sum_{v \notin A} deg(v) - \frac{(n-|A|)(p-1)}{2} \leq \sum_{v \in A} deg(v) + m[V-A] \leq m(G), \]
which follows from the fact
\[ \frac{1}{2}\sum_{v \notin A} deg(v) - \frac{(n-|A|)(p-1)}{2} = \frac{1}{2}\sum_{v \notin A} (deg(v)-(p-1)) \leq m[V-A] \]
because each of the $n-|A|$ vertices in $V-A$ has at most $p-1$ edges going back to $A$ since $G$ is $K_{1,p}$-free. Hence, $|A|$ is an integer satisfying the definition of $w(G)$, and since $w(G)$ is the largest such integer, we know $|A| \leq w(G)$. Finally, from Lemma \ref{ann_set}, we get that $|A| = a_1(G) \leq w(G)$ which completes the proof.
\end{proof}

Theorem \ref{claw} together with Theorem \ref{main} give us the corollary below. 

\begin{cor}\label{w}
Let $p\geq 3$ be an integer and let $G=(V,E)$ be a $K_{1,p}$-free graph with degree sequence $D=\{d_1 \leq d_2 \leq \ldots \leq d_n\}$.  Then, with $w(G)$ defined as above, $\alpha(G) \leq w(G)$.
\end{cor}

In a paper from 1992 by Faudree, Gould, Jacobson, Lesniak, and Lindquester \cite{11}, it is established that, for $K_{1,p}$-free graphs of order $n(G)$ with minimum degree $\delta(G)$ and independence number $\alpha(G)$,
\[ \alpha(G) \leq \frac{(p-1)n(G)}{\delta(G)+p-1}. \]
It turns out that this result follows from Corollary \ref{w}, and hence also from Theorems \ref{main} and \ref{claw}.

\begin{cor}
Let $p\geq 3$ be an integer and let $G=(V,E)$ be a $K_{1,p}$-free graph with degree sequence $D=\{d_1 \leq d_2 \leq \ldots \leq d_n\}$.  Then, with $w(G)$ defined as above, 
\[ \alpha(G) \leq a_1(G) \leq w(G) \leq \frac{(p-1)n(G)}{\delta(G)+p-1}. \]
\end{cor}

\begin{proof}
The first two inequalities in the chain have already been established, so it remains to show the final inequality.  To start, we know from the definition of $w(G)$ that
\[ m(G) \geq \sum_{i=1}^{w}d_i + \frac{1}{2}\sum_{i=w+1}^n d_i - \frac{(n-w)(p-1)}{2}. \]
Hence it is also true that,
\[ 2m(G) \geq 2\sum_{i=1}^{w}d_i + \sum_{i=w+1}^n d_i - (n-w)(p-1), \]
which is equivalent to,
\[ 2m(G)+(n-w)(p-1) \geq \sum_{i=1}^n d_i + \sum_{i=1}^{w}d_i. \]
This last inequality is equivalent to,
\[ (n-w)(p-1) \geq \sum_{i=1}^{w}d_i. \]
Finally, since
\[ \sum_{i=1}^{w}d_i \geq \delta w, \]
we arrive at
\[ (n-w)(p-1) \geq \delta w, \]
which, after rearranging, yields our desired inequality, completing the proof.
\end{proof}

We can use the DSI strategy as well to generalize this result from independence to $j$-independence.  We present a shorter direct proof below.

\begin{thm}
Let $j$ and $p\geq 3$ be integers and let $G=(V,E)$ be a $K_{1,p}$-free graph with minimum degree $\delta(G) \geq j-1$. Then, 
\[ \alpha_j(G) \leq \frac{j(p-1)n(G)}{j(p-1)+\delta(G)-(j-1)}. \]
\end{thm}

\begin{proof}
Let $A$ be a maximum $j$-independent set and denote by $m(A,V-A)$ the number of edges with one vertex in $A$ and the other vertex in $V-A$.  Since the maximum degree in $[A]$ is at most $j-1$, each vertex of $A$ has at least $\delta(G) -(j-1)$ neighbors in $V-A$. Hence,
\begin{equation}\label{first equation}
m(A,V-A) \geq |A|(\delta(G)-(j-1)) = \alpha_j(G)(\delta(G)-(j-1)).
\end{equation}

On the other hand, suppose there is a vertex $u \in V-A$ which has at least $j(p-1)+1$ neighbors in $A$.  Let $N(u)$ denote the neighbors of $u$ in $A$.  Consider the subgraph induced by $N(u)$, which we will denote by $[N(u)]$. Since $\Delta([N(u)]) \leq j-1$, there must be an independent set in $[N(u)]$ of size at least,
\[ \frac{|N(u)|}{\Delta([N(u)])+1} \geq \frac{|N(u)|}{(j-1)+1} = \frac{|N(u)|}{j} \geq \frac{j(p-1)+1}{j} = p - \frac{j-1}{j}. \]
This means that, since the independence number is an integer, $\alpha([N(u)]) \geq p$.  However, this is a contradiction since $G$ was $K_{1,p}$-free. Therefore, every vertex in $V-A$ has at most $j(p-1)$ neighbors in $A$. From this we deduce that,
\begin{equation}\label{second equation}
m(A,V-A) \leq j(p-1)(n(G)- \alpha_j(G)).
\end{equation}
Now combining Equations \ref{first equation} and \ref{second equation}, we have;
\[ \alpha_j(G)(\delta(G)-(j-1)) \leq j(p-1)(n(G)- \alpha_j(G)). \]
Solving this last inequality for $\alpha_j(G)$, we reach our desired conclusion.

\end{proof}

\section{Other Applications of the DSI Strategy}

As another example of the DSI strategy, we observe that similar treatment could be given for $j$-domination number. As was stated in the introduction, researching the $j$-domination number is very popular and some examples are \cite{29,27,20,30}.  Meanwhile, there is some strong relationships between the $j$-domination and $j$-independence numbers as seen for instance in \cite{16,24,18,17}. 

\begin{defn}\label{upper_lower_2}
Let $D=\{d_1 \leq d_2 \leq \ldots \leq d_n\}$ be the degree sequence of a graph $G=(V,E)$ and let $F$ denote the family of all minimum $j$-dominating sets in $G$.  We define the following two graph invariants;
\[ z_j(G) = \max \{ k \in Z | \sum_{i=1}^k d_i + \max_{S \in F} \{m[V-S]-m[S] \} \leq m(G) \}, \]
\[ w_j(G) = \min \{ k \in Z | \sum_{i=1}^k d_{n-i+1} + \min_{S \in F} \{m[V-S]-m[S] \} \geq m(G) \}. \]
\end{defn}

Now we have an analog to Theorem \ref{main} with respect to the $j$-domination number.

\begin{thm}\label{main2}
For any positive integer $j$ and for any graph $G=(V,E)$; 
\[
w_j(G) \leq \gamma_j(G) \leq z_j(G).
\]
\end{thm}

\begin{proof}
First we prove the upper bound.  Let $D$ be a minimum $j$-dominating set such that for all $S \in F$, $m[V-D]-m[D] \geq m[V-S]-m[S]$. Denote by $m_1$ the number of edges in $[D]$, by $m_2$ the number of edges in $[V-D]$, and by $m_3$ the number of edges between $D$ and $V-D$. Observe the following chain of inequalities:
\[
\sum_{i=1}^{\gamma_j} d_i + (m[V-D]-m[D]) \leq \sum_{v \in D} deg(v) + m_2-m_1 = 2m_1+m_3 + m_2-m_1 = m.
\]
Since $\gamma_j$ is an integer satisfying the condition in Definition \ref{upper_lower_2} above, and $z_j$ is the largest such integer, the upper bound is proven.

Next we prove the lower bound.  Let $D$ be a minimum $j$-dominating set such that for all $S \in F$, $m[V-D]-m[D] \leq m[V-S]-m[S]$. Denote $m_1$, $m_2$, and $m_3$ as above. Observe the following chain of inequalities:
\[
\sum_{i=1}^{\gamma_j} d_{n-i+1} + (m[V-D]-m[D]) \geq \sum_{v \in D} deg(v) + m_2-m_1 = 2m_1+m_3 + m_2-m_1 = m.
\]
Since $\gamma_j$ is an integer satisfying the condition in Definition \ref{upper_lower_2} above, and $w_j$ is the smallest such integer, the lower bound is proven.
\end{proof}

From this new starting point, we can repeat some of the same ideas we had for $j$-independence number.  Namely, try to approximate it in some computationally efficient way, make some structural assumptions to see what more can be said under certain conditions, and compare to known results about $j$-domination number. To give just one example of such endeavors, while choosing to leave the rest for future work, consider the following theorem.

\begin{thm}
For any positive integer $j$ and for any graph $G=(V,E)$; 
\[
\gamma_j(G) \geq w^{\prime}_j(G) = \min \{ k \in Z | \sum_{i=1}^k d_{n-i+1} + \frac{1}{2}\sum_{i=k+1}^n (d_{n-i+1}-j) \geq m(G) \}.
\]
\end{thm}

\begin{proof}
Let $D$ be a minimum $j$-dominating set and consider the following simplifying notation; $m[D]=m_1$, $m[V-D]=m_2$, and $m(D,V-D)=m_3$. To prove the theorem, we verify that $\gamma_j(G)$ is an integer satisfying the condition in the definition of $w^{\prime}_j(G)$, which is itself the smallest such integer.  To this end, it suffices to show;
\[
\sum_{i=1}^{\gamma_j} d_{n-i+1} + \frac{1}{2}\sum_{i=\gamma_j+1}^n (d_{n-i+1}-j) \geq \sum_{v \in D} deg(v) + \frac{1}{2}\sum_{v \in V-D} (deg(v)-j) \geq m(G).
\]
The first inequality above is true because any degree in $D$ not among the highest $\gamma_j(G)$ degrees in $G$ is counted with a weight of $1$ on the left but only a weight of $\frac{1}{2}$ on the right. To see that the second inequality above is true, notice that,
\[
\sum_{v \in D} deg(v) + \frac{1}{2}\sum_{v \in V-D} (deg(v)-j) = 2m_1 + m_3 + m_2 + \frac{m_3}{2} - \frac{j(n-\gamma_j)}{2}.
\]
However, since $m(G)=m_1 + m_2 + m_3$, the second inequality is true if and only if,
\[
m_1 + \frac{m_3}{2} \geq \frac{j(n-\gamma_j)}{2},
\]
or equivalently,
\[
\sum_{v \in D} deg(v) = 2m_1 + m_3 \geq j(n - \gamma_j).
\]
Finally, this last inequality is true since each of the $(n - \gamma_j)$ vertices from $V-D$ have at least $j$ neighbors in $D$ because it is a $j$-dominating set.

\end{proof}

\section{Final Remarks}

To conclude, the main goal of our paper was to present the DSI strategy and give some examples of how it could be used to derive approximations for computationally difficult graph invariants.  We showed how, as more information is known about the graph, stronger results can be obtained -- and we gave some examples of how this is done.  The authors hope there will be many other instances where the DSI strategy can be used to get new results or give deeper insight to known results. In particular, we primarily focused our studies on $k$-independence number, while leaving mostly unexplored the applications of DSI to the last parts of the paper dealing with $k$-domination number. Finally, we left open the question of whether or not equality can be obtained in Theorem \ref{planar} for the $(j,\delta)=(2,5)$ and $(j,\delta)=(4,5)$ cases.

\end{document}